\documentclass[a4paper,10pt]{amsart}

\pdfoutput=1

\usepackage[utf8]{inputenc}

\usepackage[T1]{fontenc}

\usepackage[top=3cm, bottom=3cm]{geometry}

\usepackage{lmodern}

\usepackage{microtype}

\usepackage[abbrev]{amsrefs}

\usepackage{esint}
\usepackage[pdftitle={Approximation in Sobolev spaces by piecewise affine interpolation}, pdfauthor={Jean Van Schaftingen}]{hyperref}
\title[Approximation in Sobolev spaces by interpolation]{Approximation in Sobolev spaces by piecewise affine interpolation}

\author{Jean Van Schaftingen}
\address{
  Institut de Recherche en Math\'ematique et Physique\\
  Universit{\'e} catholique de Louvain\\
  Chemin du Cyclotron 2 bte L7.01.01, 1348 Louvain-la-Neuve, Belgium}
\email{Jean.VanSchaftingen@uclouvain.be}

\newcommand{\Rset}{\mathbb{R}}

\newcommand{\norm}[1]{\left\| #1 \right\|}		
\newcommand{\abs}[1]{\lvert #1 \rvert}			
\newcommand{\Bigabs}[1]{\Bigl\lvert #1 \Bigr\rvert}			

\DeclareMathOperator{\diam}{diam}
\newcommand{\Lin}{\mathcal{L}}
\newcommand{\st}{\;:\;}
\newcommand{\dif}{\;\mathrm{d}}

\newtheorem{proposition}{Proposition}[section]
\newtheorem{theorem}{Theorem}
\newtheorem{lemma}[proposition]{Lemma}

\theoremstyle{remark}



\begin{document}

\begin{abstract}
Functions in a Sobolev space are approximated directly by piecewise affine interpolation in the norm of the space. 
The proof is based on estimates for interpolations and does not rely on the density of smooth functions.
\end{abstract}

\maketitle

\section{Introduction}

Functions in classical Sobolev spaces \(W^{1, p} (\Rset^n)\) of weakly differentiable functions, for \(p \in [1, \infty[\) can be approximated with respect to the Sobolev norm 
\[
  \norm{u}_{W^{1, p} (\Rset^n)} = \Bigl(\int_{\Rset^n} \abs{D u}^p + \abs{u}^p \Bigr)^\frac{1}{p}
\]
by various class of nicer functions.
In order to study Sobolev functions as generalizations of smooth functions, it is natural to approximate by \emph{smooth Sobolev functions} \(C^1 (\Rset^n) \cap W^{1, p} (\Rset^n)\) \cite{MeyersSerrin1964} (see also \citelist{\cite{Friedrichs1944}\cite{Ziemer}*{lemma 2.1.3}\cite{AdamsFournier2003}*{theorem 3.17}}) or by \emph{compactly supported smooth function} (see for example \citelist{\cite{Brezis2011}*{theorem 9.2}\cite{Willem2013}*{theorem 6.1.10}\cite{AdamsFournier2003}*{theorem 3.22}\cite{Tartar2007}*{lemma 6.5}}). In such a way, many properties of Sobolev functions can be proved first by differential calculus argument for smooth maps and then extended by density.

In the context of numerical resolution of partial differential equations by simplicial finite elements methods, a natural class of nice functions is the space of \emph{piecewise affine functions}. These functions are dense in the Sobolev space \(W^{1, p} (\Rset^n)\) \cite{EkelandTemam1976}*{proposition 2.8} (see also \citelist{\cite{Dacorogna2008}*{theorem 12.15}\cite{Ciarlet1978}*{theorem 3.2.3}}).

The usual proof for this statement consists in proving that piecewise affine functions approximate function in a dense subset of \(W^{1, p} (\Omega)\). The latter set can be a set of smooth functions or a higher-order Sobolev space. Sharp bounds on this approximation of smoother functions, which are known as Bramble--Hilbert lemmas, play an important role in the mathematical study of the convergence of finite element methods \citelist{\cite{Ciarlet1978}\cite{BrennerScott2008}}.

This approach in two steps is conceptually disappointing because its solves the problem of approximating by piecewise affine functions by relying on the approximation by smooth functions which is not a priori simpler, and because the diagonal argument hides the construction of the approximation: the approximating functions are piecewise affine functions whose value at vertices of the triangulation are averages in a neighbourhood of the points, and the scale of the averaging and of the triangulation need not be the same.

\bigbreak

The goal of this note is to provide a \emph{direct} approximation of Sobolev functions by \emph{interpolation}.
Given a function \(u : \Rset^n \to \Rset\) and a triangulation \(\mathcal{S}\) of the Euclidean space \(\Rset^n\) in nondegenerate \((n + 1)\)--simplices, the \emph{affine interpolant} with respect to \(\mathcal{S}\) (also known as \emph{Lagrange interpolant}) of \(u\) is the function \(\Pi_{\mathcal{S}} u : \Rset^n \to \Rset\) such that for every \((n + 1)\)--simplex \(\Sigma \in \mathcal{S}\), 
its restriction \(\Pi_{\mathcal{S}} u \vert_{\Sigma}\) is affine and at every vertex \(a\) of \(S\) the values coincide: \((\Pi_{\mathcal{S}} u) (a) = u (a) \).

We shall prove an approximation result that covers the classical Sobolev spaces as well as homogeneous Sobolev spaces.

\begin{theorem}
\label{theoremMain}
Let \(u  : \Rset^n \to \Rset\). If \(u \in L^q (\Rset^n)\), if \(u\) is weakly differentiable and if \(D u \in L^p (\Rset^n)\), then for every \(\varepsilon > 0\) there exists a triangulation \(\mathcal{S}\) of \(\Rset^n\) such that 
\[
  \int_{\Rset^n} \abs{D (u - \Pi_\mathcal{S} u)}^p + 
  \int_{\Rset^n} \abs{D (u - \Pi_\mathcal{S} u)}^q \le \varepsilon.  
\]
\end{theorem}

In order to have a well-defined interpolation \(\Pi_\mathcal{S} u\), we assume in this statement that the function \(u\) is \emph{defined at every point} of the space \(\Rset^n\), that is, we do not consider \(u\) as an equivalence class in the Lebesgue space \(L^p (\Rset^n)\). It will appear in the proof that the vertices of the triangulation will be Lebesgue points of the function \(u\).

The weakness of such a statement is the dependence of the triangulation \(\mathcal{S}\) on the function \(u\). 
Due to the minimal regularity assumptions on the function \(u\), this dependence is unavoidable.
The reader will see in the proof that the triangulation is obtained by dilation and translation from a fixed dilation and that a large set of triangulations can be used to approximate a given function \(u\).

\section{Proof of the theorem}

In order to prove theorem~\ref{theoremMain}, we study representation formulas of the affine interpolant. We begin by a Sobolev integral representation formula on a simplex.

Given a convex set \(C \subset \Rset^n\) and point \(a \in C\), the associated Minkowski gauge \(\gamma^C_a : \Rset^n \to [0, \infty]\) is defined \citelist{\cite{Clarke2013}*{p.\thinspace{}40}\cite{Rockafellar1970}*{p.\thinspace{}28}} for every \(x \in \Rset^n\) by 
\[
  \gamma^C_a (x) = \inf \bigl\{\lambda \in (0, \infty) \st a + \lambda^{-1} (x - a) \in C\bigr\}.
\]
We observe that \(\gamma^C_a (x) \le 1\) if and only if \(x\) lies in the closure of \(C\).

\begin{lemma}[Sobolev integral formula on a convex set]
\label{lemmaConvexIntegralRepresentation}
Let \(C \subset \Rset^n\) be a compact convex set with nonempty interior, let \(a \in C\) and let \(u \in L^1 (C)\). If  the function \(u\) is weakly differentiable, if \(a\) is a Lebesgue point of \(u\) and if
\[
  \int_{C} \frac{\abs{D u}}{(\gamma^C_a)^{n - 1}} < \infty,
\]
then 
\[
  u (a) -  \fint_{C} u = \frac{1}{n} \fint_{C} D u (x) [a - x]\Bigl(\frac{1}{\gamma^C_a (x)^{n}} - 1\Bigr)\dif x.
\]
\end{lemma}

We will be interested in the case where \(C\) is a simplex and \(a\) is one of its vertices. 
When \(C\) is a ball of centre \(a\), the formula is the classical Sobolev representation formula 
\[
  u (a) -  \fint_{B_R (a)} u = \frac{1}{n} \fint_{B_R (a)} D u (x) [a - x]\,\Bigl(\frac{R^n}{\abs{a - x}^{n}} - 1\Bigr)\dif x.
\]
When \(n = 1\) and  \(C = [a, b]\) with \(b \ge a\), the formula reduces to 
\[
  u (a) - \fint_{[a, b]} u = \fint_{[a, b]} u' (x) (b - x) \dif x,
\]
which can be obtained directly by integration by parts.

\begin{proof}%
[Proof of lemma~\ref{lemmaConvexIntegralRepresentation}]
Without loss of generality, we assume by translation that \(a = 0\).
We define the function \(f : [0, 1] \to \Rset\) for each \(t \in [0, 1]\) by 
\[
  f (t) = \fint_{t C} u.
\]
By a change of variable, for every \(t \in (0, 1]\),
\[
  f (t) = \fint_{C} u (t y) \dif y.
\]
Since \(a\) is a Lebesgue point of \(u\) and \(\abs{D u}/(\gamma^C_a)^{n - 1}\) is integrable, it follows that \(f\) is absolutely continuous and for almost every \(t \in (0, 1]\),
\[
  f' (t) = \int_{C} D u (t y)[y]\dif y = \frac{1}{t} \int_{tC} D u (x)[x]\dif x.
\]
Since \(t C = \{ x \in C \st \gamma^C_a (x) \le t\}\), we have by Fubini's theorem
\[
\begin{split}
  u (a) - \fint_{C} u &= f (0) - f (1) = -\int_0^1 f' (t)\\
  &= \int_0^1 \frac{1}{t} \int_{t C} D u (x)[a - x]\dif x \dif t\\
  &= \int_{C} D u (x)[a - x]\Bigl(\int_{\gamma^C_a (x)}^1 \frac{1}{t^{n + 1}}\dif t\Bigr) \dif x\\
  &= \frac{1}{n}\int_{C} D u (x) [a - x]\Bigl(\frac{1}{\gamma^C_a (x)^{n}} - 1\Bigr)\dif x.\qedhere
\end{split}
\]
\end{proof}

If \(S \subset \Rset^n\) is a simplex with vertices \(a_0, \dotsc, a_n\), then the \emph{barycentric coordinates} of a point \(x \in C\) \cite{Rockafellar1970}*{p.\thinspace{}7} are the unique real numbers \(\beta_0 (x), \dotsc, \beta_n (x) \in \Rset\) such that 
\begin{align*}
  \beta_0 (x) + \dotsb + \beta_n (x) &= 1&
  &\text{ and }&
  \beta_0 (x)a_0 + \dotsb + \beta_n (x)a_n &= a.
\end{align*}
We observe that \(\gamma^S_{a_i} = 1 - \beta_i (x)\) and we deduce directly from  lemma~\ref{lemmaConvexIntegralRepresentation} an integral representation of the derivative of the interpolant.

\begin{lemma}[Integral representation of the derivative of the interpolant]
\label{lemmaRepresentationDerivative}
Let \(\Sigma \subset \Rset^n\) be a \(n + 1\) nondegenerate simplex of vertices \(a_0, \dotsc, a_{n + 1}\) and let \(u \in L^1 (\Sigma)\) and let \(\Pi_\Sigma u : C \to \Rset\) be the affine interpolant of \(u\) on \(\Sigma\):
\[
  \Pi_\Sigma u = \sum_{i=0}^n u (x_i) \beta_i.
\]
If \(u\) is weakly differentiable in \(\Sigma\), if \(a_0, \dotsc, a_{n + 1}\) of \(\Sigma\) are Lebesgue points of \(u\) and if 
\[
  \sum_{i = 0}^n \int_{\Sigma} \frac{1}{(1 - \beta_i)^{n - 1}} \abs{D u} < \infty,
\]
then 
\[
  D (\Pi_\Sigma u) = \fint_{\Sigma} K_\sigma[D u],
\]
where the function \(K_\Sigma : \Sigma \times \Lin (\Rset^n; \Rset) \to \Lin (\Rset^n; \Rset)\) is defined for every \(x \in \Sigma\) and every \(\ell \in \Lin (\Rset^n; \Rset)\) by
\[
  K_\Sigma [\ell] = \sum_{i = 0}^n \Bigl(\frac{1}{(1 - \beta_i (x))^n} - 1\Bigr)\, \ell[a_i - x]\, \beta_i.
\]
\end{lemma}

When \(n = 1\) and \(\Sigma = [a, b]\) with \(a < b\), then the formula reduces to 
\[
  (\Pi_{\Sigma} u)' = \fint_{[a, b]} u' .
\]
Lemma~\ref{lemmaRepresentationDerivative} is reminiscent of bounds on error for the affine interpolant of twice differentiable mappings  when \(n \ge 2\) \cite{KristensenRindler},
\[
  \fint_{\Sigma} \abs{D u - D (\Pi_\Sigma u)} \le C_n (\diam S)^{n - 1} \sum_{i = 0}^n \fint_{\Sigma} \frac{\abs{D^2 u (x)}}{\abs{a_i - x}^{n - 1}}\dif x;
\]
the latter inequality can be deduced in fact from lemma~\ref{lemmaRepresentationDerivative} by the Hardy inequality with singularities at the vertices and the Poincar\'e inequality.

\begin{proof}%
[Proof of theorem~\ref{theoremMain}]
Let \(\mathcal{S}_*\) be a triangulation of \(\Rset^n\) in which all the simplices are the image under an isometry of a simplex taken in a finite set of simplices of diameter less than \(1\). (Any periodic triangulation would achieve this condition.)
We define for every \(r > 0\) and \(h \in \Rset^n\), the translated and dilated triangulation
\[
  \mathcal{S}^r_h = \{ r\Sigma + h \st \Sigma \in \mathcal{S}_*\}
\]
and we define the affine interpolant
\[
  v^r_h = \Pi_{\mathcal{S}^r_h} u.
\]
For every \(r > 0\), for almost every \(h \in \Rset^n\), the vertices of the triangulation \(\mathcal{S}^r_h\) are Lebesgue points of \(u\).
We have then for every \(\Sigma \in \mathcal{S}^r_h\), since for every \(y \in \Sigma\),
\[
  \abs{K_\Sigma (y)} \le \sum_{i = 0}^n \frac{C}{\abs{a_i - y}^{n - 1}r},
\]
by the integral representation of the derivative of the interpolant (lemma~\ref{lemmaRepresentationDerivative}) and by the H\"older inequality
\begin{equation*}
\begin{split}
  \int_{\Sigma} \abs{D u (x) - D v^r_h (x)}^p\dif x
  &= \int_{\Sigma} \Bigabs{\int_{\Sigma} K_\Sigma (y) \bigl[D u (x) - D u (y)\bigr] \dif y}^p \dif x\\
  &\le \int_{\Sigma} \Bigl(\int_{\Sigma} \abs{K_\Sigma (y)}\, \abs{D u (x) - D u (y)} \dif y\Bigr)^p \dif x\\
  &\le \int_{\Sigma} \int_{\Sigma} \abs{K_\Sigma (y)}\, \abs{D u (x) - D u (y)}^p \dif y \dif x \;\Bigl(\int_{\Sigma} \abs{K_\Sigma} \Bigr)^{p - 1}\\
  &\le \frac{C'}{r} \sum_{i=0}^n \int_{\Sigma} \int_{\Sigma} \frac{\abs{D u (x) - D u (y)}^p}{\abs{y - a_i}^{n - 1}}\dif x \dif y.
\end{split}
\end{equation*}
Since the simplices are generated by isometries and dilation from a finite number of simplices, the constant \(C\) can be taken to be independent of \(r\), \(h\) and \(\Sigma\).
If we define the function 
\[
  W (x) = \sup \Bigl\{ \frac{n + 1}{\abs{x - a}^{n - 1}} \st a \text{ is a vertex of the triangulation } \mathcal{S}_*\Bigr\}.
\]
we have, since all the simplices in \(\mathcal{S}_*\) have a diameter less than \(1\),
\[
\begin{split}
  \int_{\Sigma} \abs{D u (x) - D v^r_h (x)}^p\dif x
  &\le \frac{C'}{r^n} \int_{\Sigma} \int_{\Sigma} \abs{D u (x) - D u (y)}^p W \bigr(\tfrac{y - h}{r}\bigr)\dif y \dif x\\
  & \le \frac{C'}{r^n} \int_{\Sigma} \int_{B_r (x)} \abs{D u (x) - D u (y)}^p W \bigr(\tfrac{y - h}{r}\bigr)\dif y \dif x
\end{split}
\]

By summing the inequality over the simplices \(\Sigma\) of the triangulation \(\mathcal{S}^r_h\), we obtain 
\[
  \int_{\Rset^n} \abs{D u - D v^r_h}^p
  \le \frac{C'}{r^n} \int_{\Rset^n} \int_{B_r (x)} \abs{D u (x) - D u (y)}^p W \bigr(\tfrac{y - h}{r}\bigr)\dif y \dif x.
\]
We now integrate with respect to the translations \(h\) to obtain the estimate
\[
\begin{split}
  \fint_{B_r} \Bigl(\int_{\Rset^n} \abs{D u - D v^r_h}^p\Bigr) \dif h
  & \le \frac{C'}{r^n} \int_{\Rset^n} \int_{B_r (x)} \abs{D u (x) - D u (y)}^p \Bigl(\fint_{B_r} W \bigr(\tfrac{y - h}{r}\bigr)\dif h\Bigr) \dif y \dif x\\
  & \le C'' \int_{\Rset^n} \fint_{B_r (x)} \abs{D u (x) - D u (y)}^p \dif y \dif x\\
  & = C'' \fint_{B_r} \int_{\Rset^n} \abs{D u (x) - D u (x + h)}^p\dif x \dif h\\
  & \le C'' \sup_{h \in B_r} \int_{\Rset^n} \abs{D u (x) - D u (x + h)}^p\dif x.
\end{split}
\]
By the continuity and boundedness of translations in \(L^p (\Rset^n)\), we conclude that 
\[
  \lim_{r \to \infty} \fint_{B_r} \Bigl(\int_{\Rset^n} \abs{D u - D v^r_h}^p\Bigr) \dif h = 0.
\]

For the \(L^q\) norm, we observe that for every simplex \(\Sigma \in \mathcal{S}\) with vertices \(a_0, \dotsc, a_n\) and \(x \in \Sigma\),
\[
  \abs{u (x) - v^r_h (x)}^q \le \sum_{i = 0}^n \abs{u (x) - u (a_i)}^q
\]
Therefore, since for every \(i \in \{0, \dotsc, n\}\), \(\Sigma \subset B_r (a_i)\),
\[
  \int_{\Sigma} \abs{u - v^r_h}^q \le \sum_{i = 0}^n \int_{\Sigma} \abs{u - u (a_i)}^q
  \le \sum_{i = 0}^n \int_{B_r (a_i)} \abs{u - u (a_i)}^q
\]
By summing over the simplices \(\Sigma\) of the triangulation, we obtain
\[
  \int_{\Sigma} \abs{u - v^r_h}^q \le C \sum_{a \in \mathcal{A}^r_h} \int_{B_r (a)} \abs{u - u (a)}^q,
\]
where  \(\mathcal{A}^r_{h}\) is the set of vertices of the triangulation \(\mathcal{S}^r_h\).
By integration over \(h\), we deduce that 
\[
\begin{split}
  \fint_{B_r} \int_{\Rset^n} \abs{u - v^r_h}^q \dif h & \le C' \fint_{B_{r}} \int_{\Rset^n} \abs{u (x) - u (x + h)}^q\dif x \dif h\\
  &\le C' \sup_{h \in B_{r}} \int_{\Rset^n} \abs{u (x) - u (x + h)}^q\dif x,
\end{split}
\]
and thus 
\[
  \lim_{r \to \infty} \fint_{B_{r}} \int_{\Rset^n} \abs{u - v^r_h}^q \dif h = 0.
\]
In order to conclude, we take \(r > 0\) such that 
\[
  \fint_{B_r} \int_{\Rset^n} \abs{D u - D v^r_h}^p + \abs{u - v^r_h}^q \dif h \le \varepsilon,
\]
In particular, 
\[
  \Bigabs{ \Bigl\{h \in B_r \st \int_{\Rset^n} \abs{D u - D v^r_h}^p + \abs{u - v^r_h}^q \le \varepsilon\Bigr\}} > 0,
\]
and the conclusion follows.
\end{proof}

\section{Concluding remarks}

\subsection{Vector-valued functions}
Theorem~\ref{theoremMain} still holds for a  Sobolev vector field \(u : \Rset^n \to \Rset^m\). Moreover, all the constants appearing in the proofs do not depend on the dimension of the target space \(m\).

\subsection{Functions of bounded variation}
If \(u \in L^q (\Rset^n)\) is a function of bounded variation, that is, if \(Du\) is a vector Radon-measure then the proof above gives the existence of a triangulation \(\mathcal{S}\) of the space \(\Rset^n\) such that 
\begin{align*}
  \int_{\Rset^n} \abs{u - \Pi_{\mathcal{S}} u}^q & \le \varepsilon &
& \text{and}&
  \int_{\Rset^n} \abs{D (\Pi_\mathcal{S} u)} & \le C \int_{\Rset^n} \abs{D u},
\end{align*}
where the constant \(C\) only depends on the dimension of the space.

A better result would be an approximation in strict convergence:
\begin{equation*}
  \int_{\Rset^n} \abs{D (\Pi_\mathcal{S} u)} \le  \int_{\Rset^n} \abs{D u} + \varepsilon. 
\end{equation*}
The construction in the proof of theorem~\ref{theoremMain} satisfies this for \(n = 1\) but not for \(n \ge 2\). For example if \(n = 2\), we consider for \(u\) the characteristic function of the triangle of vertices \((0, 0)\), \((0, 1)\) and \((1, 0)\) together with the periodic triangulation generated by the triangle of vertices \((0, 0)\), \((0, 1)\) and \((1, 1)\) and the triangle of vertices \((0, 0)\), \((1, 0)\) and \((1, 1)\). 
It can be verified that
\begin{multline*}
  \liminf_{r \to \infty} \inf \Bigl\{ \int_{\Rset^2} \abs{D v^r_h} \st h \in \Rset^2 \text{ and the vertices of \(\mathcal{S}^r_h\)}\\
  \text{are Lebesgue points of \(u\)}\Bigr\} \ge 4 > 2 + \sqrt{2} = \int_{\Rset^2} \abs{D u}.
\end{multline*}
In this situation, the geometry of the triangulation is not adapted to the geometry of the jump part of the measure \(\abs{D u}\).
In fact, the known construction of approximation for functions of bounded variation relies on suitably rotated triangulations at different part of the domains to approximate the jump part of the derivative \(\abs{D u}\) \cite{KristensenRindler}.

\end{document}